\documentclass[a4paper,12pt,reqno]{amsart}

\usepackage{amsmath,amsfonts,amsthm,amssymb}
\usepackage{tikz}

\usepackage{geometry}
\usepackage{hyperref}

\newtheorem{theorem}{Theorem}[section]
\newtheorem{lemma}[theorem]{Lemma}
\newtheorem{corollary}[theorem]{Corollary}
\theoremstyle{definition}
\newtheorem{definition}[theorem]{Definition}
\newtheorem{notation}[theorem]{Notation}
\theoremstyle{remark}
\newtheorem{remark}[theorem]{Remark}

\newcommand{\integer}{\mathbb{Z}}
\newcommand{\gf}[1]{\ensuremath{\mathbb{F}_{#1}}}
\newcommand{\Vee}[2]{\ensuremath{\mathbb{V}_{#1}^{(#2)}}}
\DeclareMathOperator{\ord}{ord}

\newcommand{\dihedral}[2]{\ensuremath{\mathbb{D}_{#1,#2}}}
\newcommand{\olddihedral}[2]{\ensuremath{\mathcal{D}_{#1,#2}}}
\newcommand{\Dqr}{\dihedral{q}{r}}
\newcommand{\Dst}{\dihedral{s}{t}}

\newcommand{\W}[4]{\ensuremath{\mathbb{W}^{#1,#2}_{#3,#4}}}
\newcommand{\A}[3]{\ensuremath{\mathbb{A}^{#1}_{#2,#3}}}

\newcommand{\C}[2]{\ensuremath{{\mathcal{C}_{#1,#2}}}}

\newcommand{\Cst}{\C{s}{t}}

\newcommand{\Wqrst}{\W{q}rst}
\newcommand{\Arst}{\A{r}{s}{t}}
\newcommand{\HH}[3]{\ensuremath{\mathcal{H}_{#1}^{#2,#3}}}
\newcommand{\Hpqr}{\HH{p}{q}{r}}
\newcommand{\GK}[2]{\ensuremath{GK_{#1,#2}}}
\newcommand{\SK}[2]{\ensuremath{SK_{#1,#2}}}
\newcommand{\GKab}{\GK{a}{b}}
\newcommand{\SKab}{\SK{a}{b}}
\newcommand{\eps}{\varepsilon}

\newlength{\radius}

\begin{document}

\title[Generalised knot groups of square and granny knot analogues]{Distinguishing the generalised knot groups of square and granny knot analogues}

\author{Howida Al Fran}

\author{Christopher Tuffley}
\address{Institute of Fundamental Sciences, Massey University, Private Bag 11 222, Palmerston
North 4442, New Zealand}
\email{c.tuffley@massey.ac.nz}

\subjclass[2010]{Primary 57M27; Secondary 20F34}
\keywords{Knot invariants, generalised knot groups, square and granny knot analogues, wreath product}

\begin{abstract}
Given a knot $K$ we may construct a group $G_n(K)$ from the
fundamental group of $K$ by adjoining an $n$th root of the meridian
that commutes with the corresponding longitude. For $n\geq2$ these ``generalised
knot groups" determine $K$ up to reflection (Nelson and Neumann, 2008). 

The second author has shown that for $n\geq2$, the generalised knot groups of the square knot $SK$ and the granny knot $GK$ can be distinguished by counting homomorphisms into a suitably chosen finite group. We extend this result to certain generalised knot groups of square and granny knot analogues $\SK{a}b=T_{a,b}\# T_{-a,b}$, $\GK{a}b=T_{a,b}\# T_{a,b}$, constructed as connect sums of $(a,b)$-torus knots of opposite or identical chiralities. More precisely, for coprime $a,b\geq2$ and $n$ satisfying a certain coprimality condition with $a$ and $b$, we construct an explicit finite group $G$ (depending on $a$, $b$ and $n$)
such that $G_n(\SK{a}b)$ and $G_n(\GK{a}b)$ can be distinguished by counting homomorphisms into $G$. 
The coprimality condition includes all $n\geq2$ coprime to $ab$. 
The result shows that the difference between these two groups can be detected using a finite group. 
\end{abstract}

\maketitle

\section{Introduction}

Given a knot $K$ we may construct a group $G_n(K)$ from the
fundamental group of $K$ by adjoining an $n$th root of the meridian
that commutes with the corresponding longitude. Topologically, this corresponds to taking the fundamental group of the space $M_n(K)$ obtained by gluing a torus to the boundary of the exterior of $K$ by a suitably chosen map: expressing the boundary of the exterior as $\mu\times\lambda$, where $\mu$ and $\lambda$ are curves representing the meridian and longitude respectively, we use the map
$\phi:\mu\times\lambda\to S^1\times S^1$ given by
$\phi(z_1,z_2)=(z_1^n,z_2)$. These ``generalised
knot groups" are invariants of $K$, and were introduced independently
by Wada~\cite{wada-1992} and Kelly~\cite{kelly-1990} in the early 1990s. In addition to the topological definition above, due to Wada, they
admit several other definitions, including one via a Wirtinger-type
presentation. In this presentation there is a generator $x_i$ for each arc, as usual, but the usual crossing relations
$x_k=x_j^{\pm1 }x_i x_j^{\mp1}$ are replaced by relations of the form
$x_k=x_j^{\pm n}x_i x_j^{ \mp n}$. In particular, the group $G_1(K)$
is simply the fundamental group of $K $.

In 2008 Nelson and Neumann~\cite{nelson-neumann-2008} showed that for $n\geq 2$ the generalised knot group $G_n(K)$ determines the knot $K$ up to reflection. For $n=2$ the space $M_2(K)$ is a closed nonorientable manifold, and in this case they use the geometric version of the JSJ decomposition applied to this manifold to show that one can recover the knot complement. By Gordon and Luecke~\cite{gordon-luecke-1989} this determines the knot up to reflection.
For $n\geq 3$ the space $M_n(K)$ is not a manifold, and in this case Nelson and Neumann establish the result using the Scott-Swarup version of JSJ decomposition for groups.  

The generalised knot groups were first shown to carry more information about the knot than the fundamental group does by the second author~\cite{tuffley-2007}. This was done by confirming a conjecture due to Lin and
Nelson~\cite{lin-nelson-2008} that the generalised knot groups of the
square knot $SK$ and granny knot $GK$ are nonisomorphic for all
$n\geq2$, by showing that $G_n(SK)$ and $G_n(G K)$ can be
distinguished by counting homomorphisms into a suitably chosen finite group. In view of the Nelson-Neumann result that $G_n$ is a complete knot invariant up to reflection, this result shows that the difference between $G_n(SK)$ and $G_n(GK)$ can be detected using a finite group (albeit a large finite group: for example, for $n$ coprime to $30$ the group used has order $2\cdot3\cdot5^3\cdot q^{12}$, where $q$ is the least prime dividing $n$). The difference between these two groups can therefore be detected using a finite group, and so can be detected algorithmically. 

In this paper we extend Tuffley's construction to detect the difference between generalised knot groups of square and granny knot analogues built from $(a,b)$-torus knots instead of $(2,3)$-torus knots. For coprime integers $a$ and $b$ let $T_{a,b}$ be the $(a,b)$-torus knot, and when $a$ and $b$ are both positive let
\begin{align*}
\SKab & = T_{a,b}\# T_{-a,b}, & \GKab &= T_{a,b}\# T_{a,b}.
\end{align*}
Then the usual square and granny knots are \SK23\ and \GK23\ respectively. We prove the following theorem, which shows that (at least for certain $n$) the difference between $G_n(\SKab)$ and $G_n(\GKab)$ can be detected by counting homomorphisms into a suitably chosen finite group:

\begin{theorem}
\label{thm:main}
Let $a,b,n\geq2$ be positive integers such that $\gcd(a,b)=1$. Suppose that there are prime numbers  $s|a$ and $t|b$ such that $\gcd(st,n)=1$ (in particular, this holds if $\gcd(ab,n)=1$). Then there is a finite group $H$ such that 
\[
|Hom(G_n(GK_{a,b}),H)| < |Hom(G_n(SK_{a,b}),H)|.
\] 
\end{theorem}

The restriction on $n$ is there to simplify the arguments and make the underlying ideas more transparent. We expect that the result holds for all $n$, and that the techniques of Tuffley~\cite{tuffley-2007} can be adapted to handle the remaining cases where $\gcd(ab,n)>1$.

The building blocks for our target groups will be groups of the form
\[
\dihedral{q}r = (\integer_q)^{\ord_r(q)}\rtimes\integer_r,
\]
where $q$ and $r$ are distinct primes and $\ord_r(q)$ is the multiplicative order of $q$ modulo $r$; that is, the least positive integer $d$ such that $q^d\equiv 1\bmod r$. Using these we will construct our target groups as wreath products of two such groups, giving groups of the form
\[
\Wqrst = \dihedral{q}r\wr\dihedral{s}{t} = (\dihedral{q}r)^{s^{\ord_t(s)}}\rtimes\dihedral{s}t.
\]
Here $q$, $r$, $s$ and $t$ are distinct primes, with $s$ and $t$ the primes dividing $a$ and $b$, respectively, given in the statement of the theorem. We construct these groups in Section~\ref{sec:target}.

\subsection{Organisation}

Section~\ref{sec:definitions} is devoted to definitions.
We define the generalised knot groups in Section~\ref{sec:GnK}, and find presentations for the generalised knot groups of the square and granny knot analogues in Section~\ref{sec:GK-SK}. We then construct our target groups $\Wqrst$ and prove some needed results about them in Section~\ref{sec:target}. 
We study homomorphic images of the meridian and longitude of the $(a,b)$-torus knot group in \Wqrst\ in Section~\ref{sec:images}, in preparation for finally proving our main result in Section~\ref{sec:proof}. 

\section{Definitions}
\label{sec:definitions}

\subsection{Generalised knot groups}
\label{sec:GnK}

The generalised knot groups were defined independently by Kelly~\cite{kelly-1990} and Wada~\cite{wada-1992}, using different approaches.
Kelly reportedly~\cite{lin-nelson-2008} reached them via the knot quandle and Wirtinger presentations, while Wada's approach was to look for group-valued link invariants by studying representations $\rho:B_n\to\mathrm{Aut}(F_n)$. 
His goal was to find what he called \emph{shift representations} of the braid groups that are compatible with the Markov moves, and so could be used to define a group-valued invariant of closed braids and hence of links. A computer search found seven different types of shift representations, of which five were compatible with the Markov moves. The ``generalised knot groups'' are his ``group invariants of type 4''. 

The generalised knot group $G_n(K)$ admits a Wirtinger-like presentation, in which the usual crossing relations $x_k=x_j^{\pm1 }x_i x_j^{\mp1}$ are replaced by relations of the form
$x_k=x_j^{\pm n}x_i x_j^{ \mp n}$ (see Figure~\ref{fig:crossings}). However, the topological description of $G_n(K)$ as the fundamental group of the space $M_n(K)$ constructed above typically leads to a simpler presentation. Suppose that $\pi_1(S^3-K)$ has presentation 
\[
\pi_1(S^3-K) = \langle g_1,\dots,g_k|r_1,\dots,r_\ell\rangle,
\]
and let $\mu,\lambda$ be words in the generators representing the meridian and longitude respectively. Then applying the Seifert-van Kampen Theorem to $M_n(K)$ we obtain the presentation
\[
G_n(K) = \langle g_1,\dots,g_k,\nu|r_1,\dots,r_\ell,\nu^n=\mu,\nu\lambda=\lambda\nu\rangle. 
\]
Note further that, since $\mu$ and $\nu$ commute, we are not required to use zero-framed longitudes, and may freely replace $\lambda$ with $\lambda'=\mu^m\lambda$ for any $m$. We use this freedom in choosing the longitude below, in finding presentations for the generalised knot groups of the square and granny knot analogues.

\begin{figure}
\setlength{\radius}{2cm}
\begin{center}\setlength{\tabcolsep}{1cm}
\begin{tabular}{cc}
\begin{tikzpicture}[null/.style={minimum size = 0mm,inner sep=0mm}]
\foreach \x in {0,1,2,3}
    {\node (\x) at (45+\x*90:\radius) [null] {};}
\draw[->,thick] (2) -- (0);
\draw[line width=10,color=white] (3) -- (1);
\draw (3) edge[->,thick] (1);
\node [below right] at (0) {$x_k$};
\node [left] at (2) {$x_i$};
\node [right] at (3) {$x_j$};
\end{tikzpicture} &
\begin{tikzpicture}[null/.style={minimum size = 0mm,inner sep=0mm}]
\foreach \x in {0,1,2,3}
    {\node (\x) at (45+\x*90:\radius) [null] {};}
\draw[->,thick] (3) -- (1);
\draw[line width=10,color=white] (2) -- (0);
\draw (2) edge[->,thick] (0);
\node [below left] at (1) {$x_k$};
\node [right] at (3) {$x_i$};
\node [left] at (2) {$x_j$};
\end{tikzpicture} \\
$x_k= x_j^{n}x_ix_{j}^{-n}$ &
$x_k= x_j^{-n}x_ix_{j}^{n}$ 
\end{tabular}
\caption{Crossing relations for $G_n(K)$ at left- and right-handed crossings.}
\label{fig:crossings}
\end{center}
\end{figure}
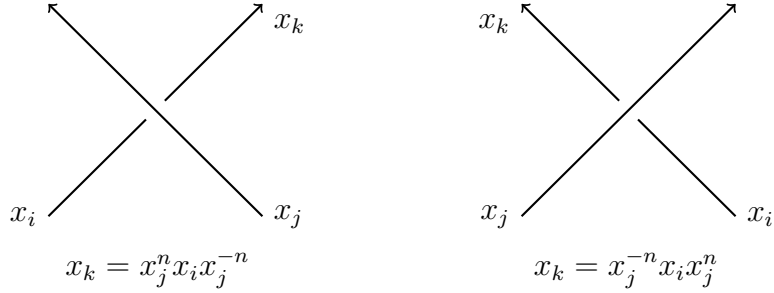

\subsection{The square and granny knot analogues}
\label{sec:GK-SK}

The square knot $SK$ is the connect sum of two trefoil knots with opposite chiralities, and the granny knot $GK$ is the connect sum of two trefoil knots with identical chiralities. Since the trefoil knot is the $(2,3)$-torus knot $T_{2,3}$, we may construct analogues of the square and granny knots by taking connect sums of $(a,b)$-torus knots instead. Accordingly, for coprime positive integers $a$ and $b$ we define the $(a,b)$-square and -granny knot analogues $\SKab$ and $\GKab$ to be the knots
\begin{align*}
\SKab & = T_{a,b}\# T_{-a,b}, & \GKab &= T_{a,b}\# T_{a,b}.
\end{align*}
Thus, the $(a,b)$-square knot analogue is the connect sum of two $(a,b)$-torus knots with opposite chiralities, and the $(a,b)$-granny knot analogue is the connect sum of two $(a,b)$-torus knots with identical chiralities.

To obtain presentations for the generalised knot groups of the square and granny knot analogues we make use of the following theorem describing the group $G_{a,b}$ of the $(a,b)$-torus knot. Of particular importance to us is part~\ref{part:meridian}, which describes the meridian and longitude.

\begin{theorem}
[See Burde, Zieschang and Heusener~{\cite[Prop 3.38 (p.~49)\footnote{Prop 3.28 (p.~45) in the original 1985 edition by Burde and Zieschang. The statement of the result differs slightly between the two editions; in particular, in the original edition the meridian is given as $u^cv^d$, where $ad+bc=1$. We have followed the newer edition.}]{burde-zieschang-heusener-2014}}]
\label{thm:Gab}
Let $(W,W')$ be a standard Heegaard splitting of genus 1 of the oriented 3-sphere $S^3$, and suppose that the torus knot $T_{a,b}$ lies on the torus $W\cap W'$.
\begin{enumerate}
\item
The group $G_{a,b}$ of the torus knot $T_{a,b}$ can be presented as follows:
\[
G_{a,b}= \langle u,v | u^a v^{-b} \rangle=\langle u|-\rangle*_{\langle u^a= v^b \rangle}\langle v |-\rangle,
\]
where $u$ is the generator of  $\pi_1(W')$, and $v$ is the generator of $\pi_1(W)$.
The amalgamating subgroup $\langle u^a\rangle$ is an infinite cyclic group; it represents the centre $Z(G_{a,b})\cong\integer$, and $G_{a,b}/Z(G_{a,b})\cong\integer_{|a|}*\integer_{|b|}$.
\item\label{part:meridian}
The elements $\mu =  v^{d}u^{-c}$, $\lambda=u^a\mu^{-ab}$, where $ad-bc=1$, describe meridian and longitude of $T_{a,b}$ for a suitably chosen basepoint.
\item
$T_{a,b}$ and $T_{a',b'}$ have isomorphic groups if and only if $|a|=|a'|$ and $|b|=|b'|$, or $|a|=|b'|$ and $|b|=|a'|$.
\end{enumerate}
\end{theorem}

By the above theorem, the peripheral system of the $(-a,b)$-torus knot is described by $(G_{-a,b},\mu,\lambda)
=(\langle u,v|u^{-a}v^{-b}\rangle,v^{d'}u^{-c'},u^{-a}\mu^{ab})$, where $(-a)d'-bc'=1$. In what follows it will be convenient to have a description that differs from that of the $(a,b)$-torus knot only in the expression for the longitude. To achieve this, let $t=v^{-1}$. Then $G_{-a,b}$ has presentation $\langle u,t|u^{-a}t^b\rangle$, and since the relation $u^{-a}t^b=1$ is equivalent to $u^at^{-b}=1$ we may write $G_{-a,b}=\langle u,v|u^at^{-b}\rangle$. The meridian is then given by $\mu=t^{-d'}u^{-c'}$, where $(-a)d'-bc'=a(-d')-bc'=1$, and setting $d=-d'$, $c=c'$ we obtain $\mu=t^d u^{-c}$ with $ad-bc=1$. Thus, replacing $t$ by $v$, the peripheral systems of $T_{a,b}$ and $T_{-a,b}$ are given by
\begin{align*}
(\pi_1(K),\mu,\lambda)=
\begin{cases}
(\langle u,v|u^av^{-b}\rangle,v^du^{-c},u^a\mu^{-ab})
 & \text{for $K=T_{a,b}$},  \\
(\langle u,v|u^av^{-b}\rangle,v^du^{-c},u^{-a}\mu^{ab})
 & \text{for $K=T_{-a,b}$},
\end{cases}
\end{align*}
differing only in the expressions for the longitudes. Since we don't require zero-framed longitudes, in what follows we will use $u^a$ as the longitude for $T_{a,b}$, and $u^{-a}$ as the longitude for $T_{-a,b}$.

To find presentations for the generalised knot groups of the square and granny knot analogues we use the fact that if $K_1,K_2$ have peripheral systems $(G_1,\mu_1,\lambda_1),(G_2,\mu_2,\lambda_2)$, then $K_1\#K_2$ has peripheral system $(G_1*_{\langle\mu\rangle}G_2,\mu,\lambda_1\lambda_2)$, where $\mu=\mu_1=\mu_2$. Let 
\[
H_{a,b}=\langle x,y,w,z | x^a = y^b, w^a = z^b, y^{d}x^{-c}=z^{d}w^{-c} \rangle.
\]
Then $\pi_1(\GK{a}b)\cong\pi_1(\SK{a}b)\cong H_{a,b}$, and the two knot groups have common meridian $\mu=y^dx^{-c}=z^{d}w^{-c}$, but different (non-zero framed) longitudes $\lambda_{\GK{a}b}=x^aw^a$ and $\lambda_{\SK{a}b}=x^aw^{-a}$.
Then presentations for $G_n(\GK{a}b)$ and $G_n(SK{a}b)$ are given by
\begin{align*}
G_n(GK_{a,b}) &\cong \langle x,y,w,z, \nu | x^a = y^b, w^a = z^b,  \nu^n =y^{d}x^{-c}=z^{d}w^{-c}, x^aw^a\nu = \nu x^aw^a\rangle,  \\
G_n(SK_{a,b}) &\cong \langle x,y,w,z, \nu | x^a = y^b, w^a = z^b,  \nu^n =y^{d}x^{-c}=z^{d}w^{-c},  x^aw^{-a}\nu =\nu  x^aw^{-a}\rangle.
\end{align*}
Observe that the two presentations differ only in the final relation, expressing the fact that the generator $\nu$ commutes with the longitude. 

\section{The target groups}
\label{sec:target}

\subsection{Introduction}
\label{sec:target.intro}

To detect the difference between the generalised knot groups of $SK=\SK23$ and $GK=\GK23$ Tuffley~\cite{tuffley-2007} used as target groups wreath products of the form 
\[
\mathcal{H}_p^{q,r} = \olddihedral{q}r\wr PSL(2,p),
\]
where $p,q,r$ are distinct primes and \olddihedral{p}q\ is a semidirect product
\[
\olddihedral{q}r = (\integer/q\integer)^{r-1}\rtimes(\integer/r\integer).
\]
The key property of $PSL(2,p)$ used was that it contains nontrivial solutions to the $(2,3)$-torus knot relation $x^3=y^2$; and moreover, any solution to this equation satisfies either $x^3=y^2=1$, or $x=z^2$, $y=z^3$ for some $z$ (Theorem~A.1 of~\cite{tuffley-2007}, proved by David Savitt). By choosing $p$ coprime to $n$ we can further ensure that there are homomorphisms $G_{2,3}\to PSL(2,p)$ such that the image of the meridian has an $n$th root. 

To detect the difference between the generalised knot groups of $\SK{a}{b}$ and $\GK{a}{b}$ we will replace $PSL(2,p)$ with  a group where a similar characterisation holds for solutions to the $(a,b)$-torus knot relation $x^a=y^b$. We will also replace $\olddihedral{q}r$ with a subgroup of itself of the form
\[
\dihedral{q}r = (\integer/q\integer)^{d}\rtimes(\integer/r\integer),
\]
where $d$ is the smallest positive integer satisfying $q^d\equiv1\bmod r$. This will ensure that $\dihedral{q}r$ is generated by any element of order $q$, together with any element of order $r$.

For distinct primes $s$ and $t$ such that $s| a$ and $t| b$ the group $\dihedral{s}{t}$ will serve as a suitable replacement for $PSL(2,p)$. Thus, our target groups will be wreath products of the form
\[
\W{q}rst=\dihedral{q}r\wr\dihedral{s}t.
\]
We describe these groups in detail below.

\subsection{The construction of the generalised dihedral groups $\dihedral{q}r$}
\label{sec:dihedral.defn}

Before constructing $\dihedral{q}r$ we first review the construction of  $\olddihedral{q}{r}$~\cite{tuffley-2007}. We introduce the following notation:

\begin{notation}
Given a prime power $\alpha$, we write $\gf{\alpha}$ for the finite field of order $\alpha$. 
\end{notation}

For $q$ and $r$ distinct primes the group $\olddihedral{q}r$ is a semidirect product 
\[
\olddihedral{q}r = (\integer_q)^{r-1}\rtimes\integer_r.
\]
To define multiplication in $\olddihedral{q}r$ we regard $(\integer_q)^{r-1}$ as the additive group of the finite field $\mathbb{F}_{q^{r-1}}$. The multiplicative group $\mathbb{F}^{\times}_{q^{r-1}}$ is cyclic of order $q^{r-1} -1\equiv 0\bmod r$, and so contains an element $\zeta$ of order $r$. Then for $i\in\integer_r$ and $v\in(\integer_q)^{r-1}$ the equation   $i\cdot v=\zeta^iv$ defines an action of $\integer_r$ on $(\integer_q)^{r-1}$ by automorphisms, and we use this action to form the semidirect product. Thus, multiplication in \olddihedral{q}r\ is defined by
\[
(v,i) \cdot (u,j) = (v+ \zeta^i u, i+j).
\]

It is shown in~\cite[Lemma~3.2]{tuffley-2007} that every element of \olddihedral{q}r\ has order $1$, $q$ or $r$. To construct \dihedral{q}r\ we will pass to a subgroup of \olddihedral{q}r\ that has the property that it is generated by any element of order $q$, together with any element of order $r$. We describe the construction of \dihedral{q}r\ independently of the construction of \olddihedral{q}r, so that it is clear that the isomorphism type depends only on $q$ and $r$.

\begin{definition}
  Given distinct primes $q$ and $r$, let $(\gf{q})^{(r)}$ be the $r$th cyclotomic field over $\gf{q}$ (the splitting field of $x^r-1$ over $\gf{q}$), and let $E_q^{(r)}\subseteq (\gf{q})^{(r)}$ be the $r$th roots of unity in $(\gf{q})^{(r)}$. Then with respect to multiplication $E_q^{(r)}$ is a cyclic group of order $r$. 
Let $\Vee{q}{r}$ be the additive group of $(\gf{q})^{(r)}$. We define the \textbf{generalised dihedral group} $\dihedral{q}r$ to be the semidirect product
\[
\dihedral{q}r = \Vee{q}{r}\rtimes E_q^{(r)},
\]
where $E_q^{(r)}$ acts on \Vee{q}{r}\ by multiplication. That is, for $v_1,v_2\in\Vee{q}{r}$ and $\zeta_1,\zeta_2\in E_q^{(r)}$, multiplication in $\Vee{q}{r}\rtimes E_q^{(r)}$ is defined by 
\[
(v_1,\zeta_1)\cdot(v_2,\zeta_2) = (v_1+\zeta_1v_2,\zeta_1\zeta_2). 
\]
\end{definition}

By~\cite[Thm~2.47]{lidl-1997} we have $[(\gf{q})^{(r)}:\gf{q}]=\ord_r(q)$, where $\ord_r(q)$ is the multiplicative order of $q$ modulo $r$; that is, the least positive integer $d$ such that $q^d\equiv 1\bmod r$. Thus
we may identify $\Vee{q}{r}$ with $(\integer_q)^{\ord_r(q)}$. Moreover, given a nontrivial choice of $\zeta\in E_q^{(r)}$ we may identify $E_q^{(r)}=\langle\zeta\rangle$ with $\integer_r$. Therefore we may regard
\dihedral{q}r\ as a semidirect product of the form
$\integer_q^{\ord_r(q)}\rtimes\integer_r$,
with multiplication defined by 
\[
(v,i) \cdot (u,j) = (v+ \zeta^i u, i+j).
\]
We necessarily have $(\gf{q})^{(r)}\leq\gf{q^{r-1}}$, so $\dihedral{q}r\leq\olddihedral{q}r$, with equality when $q$ generates $U(r)$, the group of units mod $r$.  

\begin{remark}
\label{rem:dihedral}
In the case $r=2$ we have $(\gf{q})^{(2)}=\gf{q}$, $E_q^{(2)}=\{\pm1\}$ and
$\dihedral{q}2\cong\integer_q\rtimes\integer_2\cong D_q$, the dihedral group of order $2q$ given by the symmetries of a regular $q$-gon.
\end{remark}

\subsection{Properties of $\dihedral{q}r$}

We establish some properties of $\dihedral{q}r$ that will be used in what follows. 

By construction $\Vee{q}r$ is a normal subgroup of \dihedral{q}r, and
$\dihedral{q}r/\Vee{q}r\cong E_q^{(r)}$. Choosing $1\neq\zeta\in E_q^{(r)}$ as above and identifying $E_q^{(r)}=\langle\zeta\rangle$ with $\integer_r$ via $\zeta^i\leftrightarrow i$ we therefore get a homomorphism $\dihedral{q}{r}\to\integer_r$. Given $g\in\dihedral{q}r$ we write $[g]$ for the image of $g$ in $\integer_r$ under this map; thus if $g=(v,i)\in\Vee{q}r\rtimes\integer_r$ then $[g]=i$.

Our first lemma, on elements of $\dihedral{q}r$, follows almost immediately from the corresponding result~\cite[Lemma 3.2]{tuffley-2007} for $\olddihedral{q}r$. 
\begin{lemma}
\label{lem:dihedralprops} ~

\begin{enumerate}
\item\label{item.dihedralprops.order}
If $g\in\dihedral{q}r$ then the order of $g$ is 1, $q$, or $r$. 
\item\label{item.dihedralprops.centraliser}
 If $g,h \in\dihedral{q}r$ commute, then either $g,h\in\Vee{q}r$, or $g$ and $h$ belong to the same cyclic subgroup of order $r$.  
\item\label{item.dihedralprops.conjugacy}
 If  $g = (v, 0)\in\dihedral{q}r$ has order $q$, then the conjugacy class of $g$ is 
\[
\{(\zeta^i v, 0): 0 \leq i \leq r-1\},
\] 
and if $g=(v,i)$ has order $r$ then the conjugacy class of $g$ is 
\[
\{ h : [h] = i \} = \{(v,i): v \in\Vee{q}r \}.
\] 
\end{enumerate}
\end{lemma}

\begin{proof}
Since $\dihedral{q}r\leq\olddihedral{q}r$, parts~\eqref{item.dihedralprops.order} and~\eqref{item.dihedralprops.centraliser} are immediate from~\cite[Lemma 3.2]{tuffley-2007}. Part~\eqref{item.dihedralprops.conjugacy} then follows by the same argument given there, which we now outline. By
part~\eqref{item.dihedralprops.centraliser} the centraliser of $g\in\dihedral{q}r$
has order $|\Vee{q}r|$ if $g$ has order $q$, and order $r$ if $g$ has
order $r$. Thus, the conjugacy class of $g$ has size $r$ if $g$ has
order $q$, and size $|\Vee{q}r|$ if $g$ has order $r$.
Part~\eqref{item.dihedralprops.conjugacy} now follows from the fact that the
conjugacy class of $(v,0)$ must contain the $r$ element set consisting
of the orbit of $v$ under the action of $\langle\zeta\rangle$, and the
fact that the conjugacy class of $(v,i)$ is contained in the
set $\{h:[h]=i\}$ of size $|\Vee{q}r|$.
\end{proof}

\begin{remark}
\label{rem:nthroot}
Since $q$ and $r$ are prime, as a consequence of Lemma~\ref{lem:dihedralprops} we note that if $g$ is a nontrivial element of $\dihedral{q}{r}$, then $g$ has an $n$th root in $\dihedral{q}r$ if and only if the order of $g$ is coprime to $n$. In this case the $n$th root is unique and is equal to $g^k$, where
$g^k$ is the unique $n$th root of $g$ in $\langle g\rangle$. 
In particular, $g$ has a unique $n$th root if $\gcd(qr,n)=1$. 
\end{remark}

We now show that $\dihedral{q}r$ is generated by any element of order $q$, together with any element of order $r$. 
\begin{lemma}
\label{lem:dihedral.generators}
Let $\alpha$ and $\beta$ be elements of \dihedral{q}r\ of orders $q$ and $r$, respectively. Then $\langle\alpha,\beta\rangle=\dihedral{q}r$. 
\end{lemma}

\begin{proof}
Let $\alpha=(u,0)$, and without loss of generality let $\beta=(v,1)$. Write $\beta^j=(v_j,j)$. Then 
\begin{align*}
\langle\alpha\rangle &= \{(ku, 0) : 0\leq k\leq q-1\},  \\
\langle\beta\rangle  &= \{(v_j,j) : 0\leq j\leq r-1 \}.
\end{align*}
Letting $\langle\beta\rangle$ act on $\alpha$ by conjugation we see that  
\begin{align*}
\beta^j \alpha \beta^{-j} &=(v_j, j) (u, 0) (v_j, j)^{-1} \\
                 &=(v_j, j) (u, 0) (- \zeta^{-j} v_j, -j) 
                                                  &       &\text{(since $(v_j, j)^{-1}=(-\zeta^{-j} v_j, -j)$)} \\
                  &=(v_j, j) (u- \zeta^{-j} v_j, -j) \\
                  &=(v_j+ \zeta^j(u- \zeta^{-j} v_j), j-j) \\
                  &= (\zeta^j u, 0),
\end{align*}
so $(\zeta^j u, 0)$ belongs to $\langle \alpha, \beta \rangle$ for  $0\leq j\leq q-1$.

Regarding $\Vee{q}r\cong(\integer_q)^{d}$ as a vector space over $\gf{q}$ of dimension $d=\ord_r(q)$ we claim that the vectors $u,\zeta u,\dots, \zeta^{d-1}u$ are linearly independent, and therefore span $\Vee{q}r$. To show this we use the fact that $\zeta$ is a root of the $r$th cyclotomic polynomial over \gf{q}, which since $r$ is prime and coprime to $q$ is given by
\[
Q_r(x) = \frac{x^r-1}{x-1} = \sum_{k=0}^{r-1} x^r. 
\]
Moreover, by~\cite[Thm~3.47]{lidl-1997}, $Q_r(x)$ factors over \gf{q}\ into $(r-1)/d$ monic irreducible polynomials of degree $d$. In particular, this means that the irreducible polynomial of $\zeta$ over $\gf{q}$ has degree $d$. 

Suppose now that
\begin{equation}
c_0 u+c_1 \zeta u+ \dots+ c_{d-1} \zeta^{d-1} u=0,
\label{eqn:dependence}
\end{equation}
where $c_i \in\gf{q}$ for $i=0, \dots, d-1$. Factoring~\eqref{eqn:dependence} we get
\[
\left[\sum_{i=0}^{d-1} c_i \zeta^i\right] u =0,
\]
so $\sum_{i=0}^{d-1} c_i \zeta^i = 0 \in \gf{q}$. But then if some $c_i$ is nonzero $\zeta$ is a root of $p(x)=\sum_{i=0}^{d-1} c_i x^i$, a polynomial over \gf{q}\ of degree at most $d-1$. This is a contradiction, so we must have $c_i=0$ for all $i$.  Thus $\{\zeta^i u :   0\leq i\leq d-1\}$ is a linearly independent set, as claimed. 

We have now shown that $\Vee{q}{r}\subseteq\langle\alpha,\beta\rangle$, which implies $q^d$ divides $|\langle\alpha,\beta\rangle|$. Since $r$ necessarily also divides $|\langle\alpha,\beta\rangle|$, the order of $\langle\alpha,\beta\rangle$ is divisible by $q^dr=|\dihedral{q}r|$. It follows that we must have $\langle\alpha,\beta\rangle=\dihedral{q}r$, as claimed. 
\end{proof}

A homomorphism from $G_{a,b}$ to a group $G$ corresponds to a solution to the equation $x^a=y^b$ in $G$. We make the following definition:

\begin{definition}
Let $G$ be a group, and let $a$ and $b$ be positive integers.
A solution $(x,y)=(g,h)$ to $x^a=y^b$ in $G$ is a \textbf{cyclic solution} if $\langle g,h \rangle$ is a cyclic subgroup of $G$. Otherwise, the solution is \textbf{non-cyclic}.
\end{definition} 

We now characterise when $\dihedral{q}r$ has a non-cyclic solution to $x^a = y^b$, for $a$ and $b$ relatively prime.
\begin{lemma}\label{lem:non-cyclicsolutions}
Let $a,b\geq2$ be relatively prime positive integers. Any non-cyclic solution $(x,y)=(g,h)$ to $x^a=y^b$ in $\dihedral{q}r$ satisfies $g^a=h^b=1$. Consequently, such a solution exists if and only if $q|a$ and $r|b$, or $r|a$ and $q|b$. 

\end{lemma}
\begin{proof}
Suppose that $(x,y)=(g,h)$ is a solution to $x^a = y^b$ in $\dihedral{q}r$ such that $g^a=h^b=k \neq1$.
By Lemma~\ref{lem:dihedralprops} the subgroups $\langle g\rangle$ and $\langle h\rangle$ are each cyclic of prime order $q$ or $r$, and as such they are generated by any nontrivial element. Thus $\langle g\rangle=\langle h\rangle=\langle k\rangle$, and it follows that $(g,h)$ is a cyclic solution.
Therefore any non-cyclic solution $(g,h)$ must satisfy $g^a=h^b=1$. 

Now let $(x,y)=(g,h)$ be a non-cyclic solution to $x^a=y^b$ in \dihedral{q}r. Then $g^a=1$, so $a$ must be divisible by at least one of $q$ and $r$; and similarly $h^b=1$, so $b$ must be divisible by at least one of $q$ and $r$ also. But $a$ and $b$ are relatively prime, so it must be the case that either 
$q|a$ and $r|b$, or $r|a$ and $q|b$. 

Finally, we prove that a non-cyclic solution does exist if $q|a$ and $r|b$, or $r|a$ and $q|b$.  Without loss of generality, assume $q|a$ and $r|b$. 
Let $g$ be any element of order $q$, and let $h$ be any element of order $r$. Then $g^a=1=h^b$, so $(x,y)=(g,h)$ is a solution to $x^a=y^b$; and moreover, it is a non-cyclic solution because $\langle g,h\rangle=\dihedral{q}r$, by Lemma~\ref{lem:dihedral.generators}. This completes the proof. 
\end{proof}

\begin{remark}
\label{rem:non-cyclicgenerates}
Note that by Lemmas~\ref{lem:dihedral.generators} and~\ref{lem:non-cyclicsolutions}, any non-cyclic solution $(x,y)=(g,h)$ to $x^a=y^b$ in \dihedral{q}r\ satisfies  
$\langle g,h\rangle=\dihedral{q}r$. 
\end{remark}

\subsection{The construction of the wreath products $\W{s}tqr$}
\label{sec.target.wreath}

We now describe the construction of our target groups $\W{q}rst$ as wreath products of the form $\dihedral{q}r\wr\dihedral{s}t$.

\begin{definition}
Given distinct primes $q,r,s,t$, let $\Cst\subseteq\dihedral{s}t$ be the conjugacy class
\[
\Cst=\{g\in\dihedral{s}t: [g]=1\}=\{g=(v,1)\in\dihedral{s}t:v\in\Vee{s}{t},1\in\integer_t\}.
\]
For $g\in\Cst$ and $\alpha\in\dihedral{s}t$ we write the right action of $\dihedral{s}t$ on $\Cst$ by conjugation as
\[
g\cdot\alpha = \alpha^{-1}g\alpha, 
\]
and we use this right action of $\dihedral{s}t$ on $\Cst$ to define a left action of $\Dst$ on
\[
(\dihedral{q}r)^\Cst = \{\omega=(\omega_g)_{g\in\Cst}: \omega_g\in\dihedral{q}r\}
\]
by
\[
(\alpha\cdot\omega)_g = \omega_{g\cdot\alpha}
\]
for all $\alpha\in\dihedral{s}t$ and $\omega\in(\dihedral{q}r)^\Cst$. (This is indeed a left action, because
\[
(\alpha\cdot(\beta\cdot\omega))_g  = (\beta\cdot\omega)_{g\cdot\alpha}
    = \omega_{(g\cdot\alpha)\cdot\beta} 
    = \omega_{g\cdot(\alpha\beta)} = ((\alpha\beta)\cdot\omega)_g, 
\]
and so $\alpha\cdot(\beta\cdot\omega)= (\alpha\beta)\cdot\omega$, as required.) We define $\W{q}r{s}t$ to be the semidirect product
\[
\W{q}r{s}t = \dihedral{q}r\wr\dihedral{s}t 
           = (\dihedral{q}r)^\Cst \rtimes\dihedral{s}t,
\]
where \dihedral{s}t\ acts on $(\dihedral{q}r)^\Cst$ as above.

Elements of $\W{q}rst$ have the form 
$\alpha = \bigl((\alpha_g)_{g\in\Cst},\hat\alpha\bigr)$, where
$\alpha_g\in\dihedral{q}r$ for all $g\in\Cst$ and $\hat\alpha\in\dihedral{s}t$. The group operations are given by
\begin{align*}
(\alpha\beta)_g &= \alpha_g\beta_{g\cdot\hat\alpha}, &
\widehat{\alpha\beta} &= \hat\alpha\hat\beta, \\
(\alpha^{-1})_g &= (\alpha_{g\cdot\hat\alpha^{-1}})^{-1}, &
\widehat{\alpha^{-1}} &= \hat\alpha^{-1}. 
\end{align*}

\end{definition}

\begin{remark}
\label{rem:cyclestructure}
We use Lemma~\ref{lem:dihedralprops} to describe the cycle structure of an element $\hat{\alpha} \in\Dst$ acting on $\Cst$ by conjugation.  Note that all elements of $\Cst$ have order $t$. If $\hat{\alpha}$ has order $s$, then $\alpha\in\Vee{s}t$. Elements of order $s$ and $t$ do not commute, and therefore the action of $\hat{\alpha}$ on $\Cst$ has no fixed point and all cycles are of length $s$, because $s$ is prime. If $\hat{\alpha}$ has order $t$, then it commutes only with elements of $\langle \hat{\alpha}\rangle$. 
Since $\hat{\alpha}^k \in \Cst$ for a unique $k$ satisfying $0\leq k\leq t-1$, the action of $\hat{\alpha}$ on $\Cst$ has a unique fixed point, and therefore all other cycles have length $t$, because $t$ is prime.
\end{remark}

We now describe some subgroups and homomorphisms associated with \W{q}rst\ that will be of use. The constructions parallel those of~\cite[Sec.~3.1]{tuffley-2007}. 
The quotient map $[\cdot]:\dihedral{q}r \to \mathbb{Z}_r$ induces a quotient map 
\[
\W{q}rst\to\integer_r\wr\dihedral{s}t =(\integer_r)^\Cst\rtimes\dihedral{s}t,
\] 
which is given by 
\[
[\alpha]=(([\alpha_g])_{g \in \Cst}, \hat{\alpha}).
\]
This map splits, and it will be convenient to distinguish a subgroup of \W{q}rst\ isomorphic to $\integer_r\wr\dihedral{s}t$. Fixing $\xi\in\dihedral{q}r$ of order $r$ such that $[\xi]=1$ we let
\[
\A{r}st = \langle\xi\rangle\wr\dihedral{s}t = \langle\xi\rangle^\Cst\rtimes\Dst\leq \W{q}rst.
\]
Since $\integer_r$ is abelian we may quotient further to get a well defined map 
\[
\W{q}rst\to (\integer_r)^\Cst\rtimes\dihedral{q}r \to \integer_r,
\]
given by
\[
[[\alpha]]= \sum_{g\in \Cst}[\alpha_g].
\]

\subsection{Illustration}
\label{sec:example}

As an aid to understanding the construction and arguments we illustrate the group operations in a group of the form $G=H\wr\dihedral{5}2$, where $H$ is an arbitrary group. 

As noted in Remark~\ref{rem:dihedral} we have $\dihedral{5}2=D_5$, the dihedral group of order 10, given by the symmetry group of a regular pentagon. For $k\in\integer_5$ define $\rho_k,\sigma_k:\integer_5\to\integer_5$ by
\begin{align*}
i\cdot\rho_k &= k+i, & 
i\cdot\sigma_k &= 2k-i.
\end{align*}
We write the argument to the left of the function so that composition is from left to right. Then
\[
\dihedral{5}2=D_5\cong\{\rho_k,\sigma_k:k\in\integer_5\}.
\]
In addition, $\Vee{5}2$ is the subgroup $\langle\rho_1\rangle=\{\rho_k:k\in\integer_5\}$ of rotations, and $\C52$ is the conjugacy class $\{\sigma_k:k\in\integer\}$ consisting of the reflections. We have 
\begin{align*}
\rho_k^{-1}\sigma_i\rho_k &= \sigma_{k+i} = \sigma_{i\cdot\rho_k}, \\
\sigma_k^{-1}\sigma_i\sigma_k &= \sigma_{2k-i}= \sigma_{i\cdot\sigma_k},
\end{align*}
so the action of $D_5$ on $\C52$ by conjugation may be identified with the action of $D_5$ on $\integer_5$. We may therefore regard $G=H\wr D_5$ as $H^{\integer_5}\rtimes D_5$. 

An element $\gamma$ of $G$ has the form $\bigl((\gamma_i)_{i\in\integer_5},\hat\gamma\bigr)$, where $\gamma_i\in H$ for each $i$ and $\hat\gamma\in D_5$. We may represent $\gamma$ as an edge labelled directed graph with vertex set $\integer_5$, where we draw a directed edge labelled $\gamma_i$  from vertex $i$ to vertex $i\cdot\hat\gamma$. 
Figures~\ref{fig:illustration}(i) and~\ref{fig:illustration}(ii) present such diagrams for elements $\alpha$ and $\beta$ such that $\hat\alpha=\rho_1$ and $\hat\beta=\sigma_0$.
The action of $D_5$ on $\integer_5$ is faithful, so for $\gamma\in G$ we can completely recover $\gamma$ from its associated edge labelled directed graph.  

Given group elements $\gamma$ and $\delta$ of $G$, to calculate the product $\gamma\delta$ we follow each directed edge of $\gamma$, and then follow the directed edge of $\delta$ beginning at its terminus, multiplying the labels. This gives a directed edge from $i$ to $i\cdot\widehat{\gamma\delta}= (i\cdot\hat\gamma)\cdot\hat\delta$ for each $i$, with label $\gamma_i\delta_{i\cdot\hat\gamma}$. So for example, in the diagram for $\alpha\beta$ in Figure~\ref{fig:illustration}(iii), we get an edge from $0$ to $4$ labelled $\alpha_0\beta_1$ by following the edge of $\alpha$ labelled $\alpha_0$ from $0$ to $1$, and then following the edge of $\beta$ labelled $\beta_1$ from $1$ to $4$.

Inverses are found by reversing all arrows and inverting all labels, as shown in Figure~\ref{fig:illustration}(iv) for the element $\alpha$ of  Figure~\ref{fig:illustration}(i). Note that the edge labelled $(\alpha_i)^{-1}$ now begins at $i\cdot\hat\alpha$ instead of $i$, and so $(\alpha^{-1})_i = (\alpha_{i\cdot\hat\alpha^{-1}})^{-1}$.

\begin{figure}
\setlength{\radius}{2cm}
\renewcommand{\arraystretch}{1.3}
\begin{center}
\setlength{\tabcolsep}{0.5cm}
\begin{tabular}{ll}
(i)\begin{tikzpicture}[vertex/.style={circle,draw,fill=black,minimum size = 1.5mm,inner sep=0pt},thick]
\foreach \x/\p in {0/above,1/{above left},2/{below left},3/{below right},4/{above right}}
{\node (\x) at (\x*72+90:\radius) [vertex] {};
 \node [\p] at (\x) {$\x$};}
\foreach \x in {0,1,...,4}
\node (mid\x) at (\x*72+54:\radius) {};
\foreach \x in {0,1,...,4} 
\draw[->] (mid\x) arc (\x*72+54:(\x+1)*72+53:\radius);
\foreach \x in {0,1,2,3,4}
   \node at (\x*72+126:1.2*\radius) {$\alpha_\x$};
\end{tikzpicture} & 
(ii)\begin{tikzpicture}[vertex/.style={circle,draw,fill=black,minimum size = 1.5mm,inner sep=0pt},thick,null/.style={minimum size = 0mm,inner sep=0mm}]
\foreach \x/\p in {0/{right},1/{above left},2/{below left},3/{below right},4/{above right}}
{\node (\x) at (\x*72+90:\radius) [vertex] {};
 \node [\p] at (\x) {$\x$};}
\foreach \x/\y/\p in {1/4/below,2/3/below,3/2/above,4/1/above}
   \draw (\x) edge [->,bend right] node [pos=0.75,\p] {$\beta_\x$} (\y);
\node (0label) at (90:1.5*\radius) [null] {};
\draw (0) edge [->,out=60,in=0] (0label); 
\draw (0label) edge [out=180,in=120] (0); 
\node [above] at (0label) {$\beta_0$};
\end{tikzpicture} \\
(iii)\begin{tikzpicture}[vertex/.style={circle,draw,fill=black,minimum size = 1.5mm,inner sep=0pt},thick,null/.style={minimum size = 0mm,inner sep=0mm}]
\foreach \x/\p in {0/{above},1/{left},2/{above},3/{right},4/{right}}
{\node (\x) at (\x*72+90:\radius) [vertex] {};
 \node [\p] at (\x) {$\x$};}
\foreach \x/\y/\z/\p in {1/3/2/{below},3/1/4/{above},4/0/0/{above},0/4/1/{below}}
   \draw (\x) edge [->,bend right] node [pos=0.7,\p] {$\alpha_\x\beta_\z$} (\y);
\node (2label) at (234:1.5*\radius) [null] {};
\draw (2) edge [->,out=204,in=144] (2label); 
\draw (2label) edge [out=324,in=264] (2); 
\node [below left] at (2label) {$\alpha_2\beta_3$};
\end{tikzpicture} &
(iv)\begin{tikzpicture}[vertex/.style={circle,draw,fill=black,minimum size = 1.5mm,inner sep=0pt},thick]
\foreach \x/\p in {0/above,1/{above left},2/{below left},3/{below right},4/{above right}}
{\node (\x) at (\x*72+90:\radius) [vertex] {};
 \node [\p] at (\x) {$\x$};}
\foreach \x in {0,1,...,4}
\node (mid\x) at (\x*72+54:\radius) {};
\foreach \x in {0,1,...,4} 
\draw[<-] (mid\x) arc (\x*72+54:(\x+1)*72+53:\radius);
\foreach \x/\r in {0/1.2,1/1.35,2/1.2,3/1.35,4/1.25}
   \node at (\x*72+126:\r*\radius) {$(\alpha_\x)^{-1}$};
\end{tikzpicture}
\end{tabular}
\caption{Illustrating the group operations in a group of the form
$G=H\wr\,{\dihedral52}=H\wr D_5$. Diagrams representing elements (i) $\alpha$ with $\hat\alpha=\rho_1$; (ii) $\beta$ with $\hat\beta=\sigma_0$; (iii) $\alpha\beta$; and (iv) $\alpha^{-1}$. }
\label{fig:illustration}
\end{center}
\end{figure}
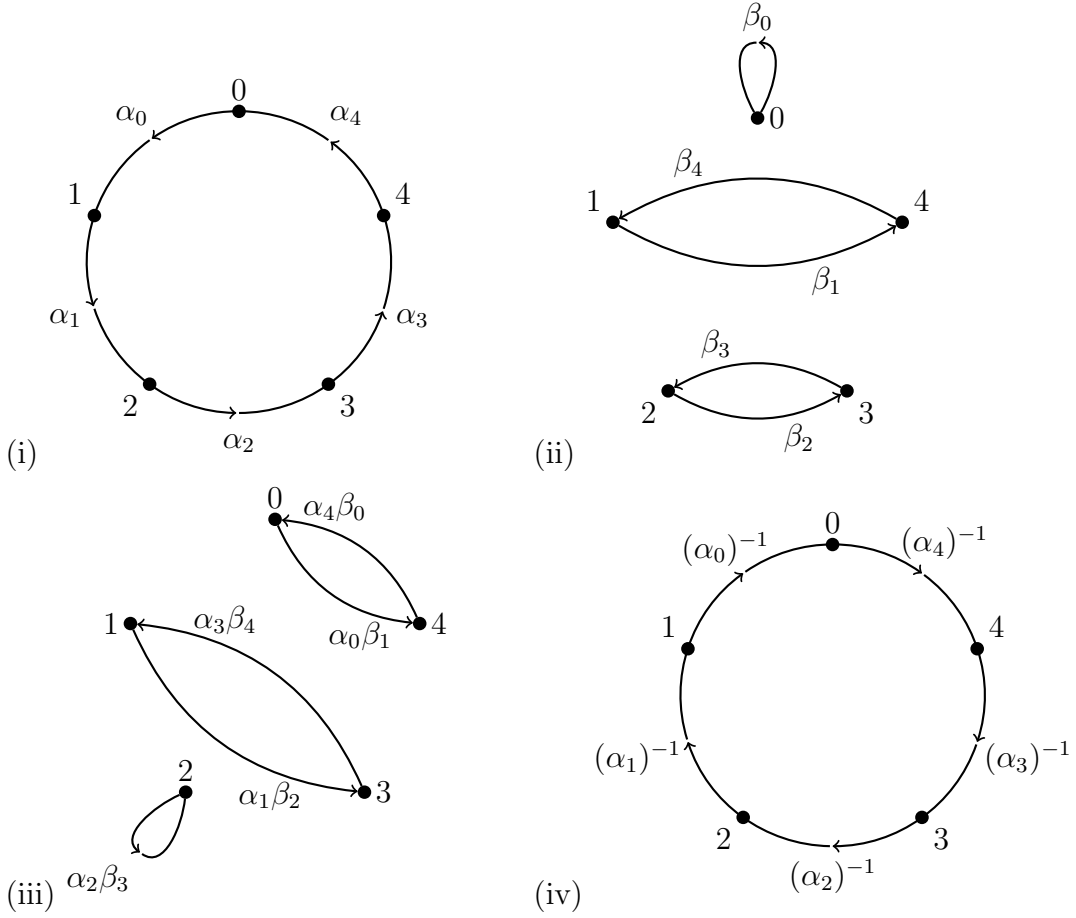

\subsection{The cycle product and applications}
\label{sec:cycleprod}

As in Tuffley~\cite{tuffley-2007}, the main tools we require to work with our target groups \Wqrst\ are the \emph{cycle product}, and the notion of an element of \Wqrst in \emph{reduced standard form}. These are useful in understanding conjugacy classes, centralisers and $m$th powers in \Wqrst. Adapting the definitions from $\Hpqr=\olddihedral{q}r\wr PSL(2,p)$ to $\Wqrst=\dihedral{q}r\wr\dihedral{s}t$ we define these as follows:

\begin{definition}[Tuffley~{\cite[Sec.~3.4]{tuffley-2007}}]
Given $\alpha\in\Wqrst$ and $g \in \Cst$, we let $\ell_g(\hat {\alpha})$ be the length of the disjoint cycle of the action of $\hat{\alpha}$ on \Cst\ that contains $g$. Then the \textbf{cycle product of $\alpha$ at $g$} is the product
\[
\pi_g(\alpha) = \prod_{r=0}^{\ell_g(\hat\alpha)-1} \alpha_{g \cdot \hat\alpha^r} = \alpha_g \alpha_{g \cdot \hat\alpha} \alpha_{g \cdot \hat\alpha^2}\cdots \alpha_{g \cdot \hat\alpha^{\ell_g (\hat\alpha)-1}}.
\]
\end{definition}
The cycle product is thus the ordered product, beginning at $g$, of $\alpha_h$ for $h$ in the disjoint cycle of $\hat{\alpha}$ containing $g$. Observe that the value of the cycle product on a given cycle is dependent on the beginning point $g$, while the conjugacy class is not, because $\pi_{h \cdot \hat{\alpha}}(\alpha) = \alpha_h^{-1} \pi_h(\alpha)\alpha_h$.

\begin{definition}[See~{\cite[Sec.~3.4]{tuffley-2007}}]
\label{def:rsf}
Let $\gamma$ belong to $\Wqrst$. Then $\gamma$ is in \textbf{reduced standard form} if 
\[
\gamma_{g \cdot \hat{\gamma}} = \gamma_g
\]
for each $g\in\Cst$. 
In other words, $\gamma$ is in reduced standard form if $\gamma_g$ is constant on orbits of $\hat\gamma$. 
\end{definition}

\begin{remark}
In~\cite{tuffley-2007}, an element satisfying the condition of Definition~\ref{def:rsf} is only said to be in \emph{standard form}, and to be in reduced standard form we further require that
$\pi_g (\gamma) = \gamma_g^{\ell_g(\hat{\gamma})} = 1$ if and only if $\gamma_g = 1$. We observe that these two notions co-incide in \Wqrst, because the orders of elements of \dihedral{q}r\ are coprime to the orders of elements of \dihedral{s}t. If 
$\pi_g (\gamma) = \gamma_g^{\ell_g(\hat{\gamma})} = 1$ then $\ell_g(\hat{\gamma})$ must divide the order of $\gamma_g$. But $\gamma_g\in\dihedral{q}r$, so the order of $\gamma_h$ is 1, $q$ or $r$; and $\ell_g(\hat{\gamma})$ is the length of a cycle of $\hat\gamma\in\dihedral{s}t$ acting on $\Cst$, so is 1, $s$ or $t$. By construction $q,r,s,t$ are distinct primes, so if  $\pi_g (\gamma) = 1$ then $\gamma_g=1$ also.
\end{remark}

We now state as lemmas several applications of the cycle product to conjugacy classes, centralisers and $m$th powers in \Wqrst. Where proofs are not given see~\cite[Sec.~3.4]{tuffley-2007}. Although the proofs given there are for 
$\Hpqr=\olddihedral{q}r\wr PSL(2,p)$, they in fact apply to any group of the form $\olddihedral{q}r\wr G=(\olddihedral{q}r)^X\rtimes G$ or $\dihedral{q}r\wr G=(\dihedral{q}r)^X\rtimes G$, where $G$ is a group and $X$ is a right $G$-set. This is because the structure of $PSL(2,p)$ plays no role in the arguments: the arguments are all carried out on the level of cycles and orbits of an element $k\in G$ acting on $X$, making use of the properties of $\olddihedral{q}r$ given by~\cite[Lemma~3.2]{tuffley-2007}, or by Lemma~\ref{lem:dihedralprops} for $\dihedral{q}r$. 

Our first result is that every element of \Wqrst\ may be conjugated to reduced standard form:
\begin{lemma}
\label{lem:rsf}
Let $\alpha\in\Wqrst$.
Then $\alpha$ is conjugate to an element $\gamma$ in reduced standard form such that $\hat{\gamma} = \hat{\alpha}$ and $\pi_g(\gamma)$ is conjugate to $\pi_g(\alpha)$ for all $g\in\Cst$.
\end{lemma}

\begin{proof}
Applying the argument of~\cite[Lemma 3.3]{tuffley-2007}, 
$\alpha$ is conjugate to such $\gamma$ as required
if $\ell_g(\hat{\alpha})$ is coprime to the order of $\pi_g(\alpha)$ for all $g\in\Cst$.  But $\ell_g(\hat{\alpha}) \in \{1,s,t\}$ and $\ord(\pi_g(\alpha)) \in \{1,q,r\}$ for all $g\in\Cst$, and by construction $q,r,s,t$ are distinct primes, so the coprimality condition is always satisfied. Thus, the conclusion of the lemma holds for all $\alpha\in\Wqrst$. 
\end{proof}

We now give some sufficient conditions for an element of \Wqrst\ to be conjugate to an element of \Arst\ in reduced standard form:

\begin{lemma}[See {\cite[Lemma 3.4]{tuffley-2007}}]
\label{lem:inArst}
Let $\alpha$ be an element of $\Wqrst$ in reduced standard form such that $\alpha_g$ has order 1 or $r$ for all $g\in\Cst$. Then $\alpha$ is conjugate to an element $\gamma$ of $\Arst$ in reduced standard form such that $\hat\gamma=\hat\alpha$ and $\gamma_g$ is conjugate to $\alpha_g$ for all $g$. 
\end{lemma} 

\begin{corollary}
\label{cor:inArst}
Let $\alpha$ be an element of $\Wqrst$ such that $\pi_g(\alpha)$ has order 1 or $r$ for all $g\in\Cst$. Then $\alpha$ is conjugate to an element $\gamma$ of $\Arst$ in reduced standard form such that $\hat\gamma=\hat\alpha$. 
\end{corollary}

\begin{proof}
By Lemma~\ref{lem:rsf}, $\alpha$ is conjugate to an element $\beta$ in reduced standard form such that $\hat\beta=\hat\alpha$ and $\pi_g(\beta)$ is conjugate to $\pi_g(\alpha)$ for all $g\in\Cst$. Thus, $\pi_g(\beta)$ has order $1$ or $r$ for all $g\in\Cst$. Because $\beta$ is in reduced standard form we have
\[
\pi_g(\beta) = \beta_g^{\ell_g(\hat\beta)}, 
\]
where $\ell_g(\hat\beta)\in\{1,s,t\}$. Since $\ord{\beta_g}\in\{1,q,r\}$, and $q$ and $r$ are prime and coprime to $st$, it must be the case that $\ord(\beta_g)=\ord(\pi_g(\beta))$, and so $\beta_g$ has order $1$ or $r$ for all $g\in\Cst$. Then by Lemma~\ref{lem:inArst} $\beta$, and hence also $\alpha$, is conjugate to an element $\gamma$ of \Arst\ in reduced standard form such that $\hat\gamma=\hat\beta=\hat\alpha$. 
\end{proof}

Next, we consider certain elements of the centraliser of an element in reduced standard form.
\begin{lemma}[See {\cite[Lemma 3.5]{tuffley-2007}}]
\label{lem:rsf-centraliser}
Let $\alpha\in\Wqrst$ be an element of \Wqrst\  in reduced standard form, and suppose that $\beta$ commutes with $\alpha$. If $\alpha_g$ is constant on orbits of $\hat{\beta}$, then $\beta_g$ commutes with $\alpha_g$ for all $g\in\Cst$, and $\beta_g$ is constant on orbits of $\hat{\alpha}$.
\end{lemma}

In particular, the condition that  $\alpha_g$ is constant on orbits of $\hat{\beta}$ holds if $\hat\beta\in\langle\hat\alpha\rangle$. Finally, we give a necessary condition for an element $\alpha$ of $\Wqrst$ to be an $m$th power:
\begin{lemma}[See {\cite[Lemma 3.6]{tuffley-2007}}]
\label{lem:mthpower}
Suppose that $\alpha = \gamma^m$ in $\Wqrst$. Then $\hat{\gamma}^m = \hat{\alpha}$ and 
\[
\pi_g(\alpha) = (\pi_g(\gamma))^{m/\gcd(\ell_g(\hat{\gamma}),m)}.
\]
In particular, the conjugacy class of $\pi_g(\alpha)$ is constant on orbits of $\hat{\gamma}$; and if $\gcd(st,m)=1$ then $\pi_g(\alpha) = (\pi_g(\gamma))^{m}$. 
\end{lemma}

\section{Homomorphisms from $G_{a,b}$ to \Wqrst}
\label{sec:images}

\subsection{Introduction}
We will prove Theorem~\ref{thm:main} in the following form:

\begin{theorem}
\label{thm:main-refined}
Let $a,b,n\geq2$ be positive integers such that $\gcd(a,b)=1$. Suppose that there are prime numbers  $s|a$ and $t|b$ such that $\gcd(st,n)=1$. Choose a prime $q$ dividing $n$, and a prime $r$ coprime to $2nab$. Then $H=\Wqrst$ satisfies the conclusion of Theorem~\ref{thm:main}; that is, 
\[
|Hom(G_n(\GK{a}{b}),\Wqrst)| < |Hom(G_n(\SK{a}{b}),\Wqrst)|.
\] 
\end{theorem}

The underlying strategy is identical to that used in~\cite{tuffley-2007} to distinguish the generalised knot groups of the square and granny knots, $SK=\SK23$ and $GK=\GK23$. Both groups $G_n(\GK{a}{b})$, $G_n(\SK{a}{b})$ are obtained from
\[
H_{a,b}=\pi_1(\GK{a}{b})=\pi_1(\SK{a}{b})=
\langle x,y,z,w | x^a=y^b, w^a = z^b,y^{d}x^{-c} = z^{d}w^{-c} \rangle
\]
by adjoining an $n$th root of the common meridian $\mu=y^{d}x^{-c} = z^{d}w^{-c}$ that commutes with the corresponding longitude 
$\lambda_{GK_{a,b}}=x^aw^a$ or $\lambda_{SK_{a,b}}=x^aw^{-a}$. Thus, to compare the number of homomorphisms from 
$G_n(\GK{a}{b})$ and  $G_n(\SK{a}{b})$ to a given finite group $H$ we consider pairs of the form $(\rho,\eta)$, where $\rho:H_{a,b}\to H$ is a homomorphism and 
$\eta\in H$ is an $n$th root of $\rho(\mu)$. We say that such a pair $(\rho,\eta)$ is a \emph{map-root pair} for $H_{a,b}$ in $H$. For $K=\GK{a}b$ or $\SK{a}b$ a map-root pair for $H_{a,b}$ in $H$ defines a homomorphism
$\tilde\rho:G_n(K)\to H$ precisely when it satisfies the \emph{compatibility condition} 
\[
\rho(\lambda) \eta = \eta \rho(\lambda). 
\]
For $\GK{a}b$ this may be written in the form
\[
\rho(x^a)\rho(w^a)\eta = \eta\rho(x^{a})\rho(w^a) \qquad\text{or}\qquad
\rho(x^{-a})\eta\rho(x^{a}) = \rho(w^a)\eta\rho(w^{-a}),
\]
while for $\SK{a}b$ it may be written as
\[
\rho(x^a)\rho(w^{-a})\eta = \eta\rho(x^{a})\rho(w^{-a}) \qquad\text{or}\qquad
\rho(x^{-a})\eta\rho(x^{a}) = \rho(w^{-a})\eta\rho(w^{a}).
\]
Thus, it suffices to prove the following:
\begin{enumerate}
\renewcommand{\theenumi}{\Roman{enumi}}
\item 
\label{item:inequality}
Every map-root pair for $H_{a,b}$ in \Wqrst\ that is compatible for $\GK{a}{b}$ is also compatible for $\SK{a}{b}$.
\item
\label{item:strict}
There exist map-root pairs for $H_{a,b}$ in \Wqrst\ that are compatible for $\SK{a}{b}$ but not $\GK{a}{b}$.
\end{enumerate}

Since $H_{a,b} = G_{a,b}*_{\langle\mu\rangle}G_{a,b}$, a homomorphism $\rho:H_{a,b}\to H$ consists of a pair of homomorphisms $\rho_1,\rho_2:G_{a,b}\to H$ such that $\rho_1(\mu)=\rho_2(\mu)$. We therefore begin by determining the possible images of the meridian and longitude of $G_{a,b}$ under a homomorphism $G_{a,b}\to\Wqrst$. The results and proofs of this section parallel those of~\cite[Sec.\ 4]{tuffley-2007}.

\subsection{The image and roots of the meridian}
\label{sec:meridian}

We characterise up to conjugacy solutions to the equation $\eta^n=\alpha$ in \Wqrst, with $\hat\alpha\neq1$. 

\begin{lemma}
\label{lem:meridian}
Let $n\geq2$ be a positive integer, and let $q,r,s,t$ be distinct primes such that $q|n$ and $\gcd(rst,n)=1$. If $\alpha\in\Wqrst$ is an $n$th power such that $\hat{\alpha} \neq 1$, then $\alpha$ is a conjugate to an element of $\Arst$ in reduced standard form. 
Conversely, if $\alpha\in\Arst$ is in reduced standard form and $\hat\alpha\neq1$ then the solutions to $\eta^n=\alpha$ in \Wqrst\ are described by
\begin{enumerate}
\item\label{item:shape}
$\hat{\eta}=\hat{\alpha}^k$,
where $\hat{\alpha}^k$ is the unique $n$th root of $\hat{\alpha}$ in $\langle \hat{\alpha} \rangle$;
\item\label{item:eta-is-standard}
$\eta_{g\cdot\hat\alpha}=\eta_g$ for all $g\in\Cst$, and so also 
$\eta_{g\cdot\hat\eta}=\eta_g$ for all $g\in\Cst$;
\item\label{item:alphagnontrivial}
$\eta_g=\alpha_g^{1/n}\in\langle\xi\rangle$
if $\alpha_g \neq 1$, where $1/n$ is the multiplicative inverse of $n$ in $\mathbb{Z}_r$; and
\item\label{item:alphagtrivial}
$\eta_g\in \Vee{q}r$ if $\alpha_g = 1$.
\end{enumerate}
Consequently, the solutions to $\eta^n=\alpha$ in \Wqrst\ are parameterised by $(\Vee{q}{r})^c$, where $c$ is the number of cycles of $\hat{\alpha}$ on which $\alpha_g=1$.
\end{lemma}

\begin{proof}
Suppose $\eta^n=\alpha$. Then by Lemma~\ref{lem:mthpower} we have 
\[
\pi_g(\alpha) = (\pi_g(\eta))^{n}
\]
for all $g\in\Cst$. 
Now $q$ divides $n$, so $\pi_g(\alpha)\in\dihedral{q}r$ is a $q$th power in $\dihedral{q}r$ and therefore $\pi_g(\alpha)$ has order 1 or $r$ for all $g$.
Therefore by Corollary~\ref{cor:inArst} $\alpha$ is conjugate to an element of $\Arst$ in reduced standard form, as claimed.

Now let $\alpha$ be an element of $\Arst$  in reduced standard form such that $\hat\alpha\neq 1$, and suppose that $\eta\in\Wqrst$ satisfies $\eta^n=\alpha$. Then we must have $\hat\eta^n=\hat\alpha$ in \dihedral{s}t, 
and since $\gcd(st,n)=1$, by Remark~\ref{rem:nthroot} $\hat\eta=\hat\alpha^k$, 
where $\hat{\alpha}^k$ is the unique $n$th root of $\hat{\alpha}$ in $\langle \hat{\alpha} \rangle$. This establishes part~\eqref{item:shape}. 

The order of $\hat{\alpha}$ is prime, so the orbits of $\hat{\eta}=\hat{\alpha}^k$ and the orbits of $\hat{\alpha}$ coincide. Therefore $\alpha_g$ is constant on orbits of $\hat{\eta}$, because $\alpha$ is in reduced standard form. Since $\alpha=\eta^n$ commutes with $\eta$, Lemma~\ref{lem:rsf-centraliser} implies that $\alpha_g$ commutes with $\eta_g$ for all $g\in\Cst$, and that $\eta_g$ is constant on the orbits of $\hat\alpha$. Thus $\eta_{g\cdot \hat{\alpha}}=\eta_g$ for all $g\in\Cst$, and then also $\eta_{g \cdot \hat{\eta}}=\eta_g$ for all $g$, because $\hat\eta=\hat\alpha^k$. This proves part~\eqref{item:eta-is-standard}. 

We have now shown that $\eta$ is in reduced standard form. We therefore have
\[
\alpha_g = (\eta^n)_g = (\eta_g)^n
\]
in $\dihedral{q}r$. By hypothesis $n$ is divisible by $q$ but not by $r$.
Therefore, by Remark~\ref{rem:nthroot}
either $\alpha_g$ is of order $r$ and $\eta_g$ is the unique $n$th root of $\alpha_g$ in $\langle\alpha_g\rangle=\langle\xi\rangle$, or $\alpha_g=1$ and then we can let $\eta_h=\eta_g=v$ for all $h$ in the orbit of $g$ under $\hat\eta$, for any $v\in\Vee{q}r$. This proves parts~\eqref{item:alphagnontrivial} and~\eqref{item:alphagtrivial}. 
These values (with $v$ chosen independently on each orbit where $\alpha_g=1$) together with $\hat\eta=\hat\alpha^k$ do in fact define $n$th roots, so the $n$th roots of  $\alpha$  are parameterised by $(\Vee{q}r)^c$, where $c$ is the number of cycles of $\hat{\alpha}$ on which $\alpha_g=1$.
\end{proof}

\subsection{The image of the longitude}
\label{sec:longitude}

We now determine the form of the image of the longitude $x^a=y^b$, under the assumption that the meridian maps to an element $\alpha$ of the form considered in Lemma~\ref{lem:meridian}. 

\begin{lemma}
\label{lem:longitude}
Let $a,b,n\geq2$ be positive integers such that $\gcd(a,b)=1$, and let $q,r,s,t$ be distinct primes satisfying the conditions of Theorem~\ref{thm:main-refined}. 
Let $\alpha$ be an element of $\Arst$ in reduced standard form such that $\hat {\alpha}$ is nontrivial.  
Suppose that $\rho: G_{a,b} \to\Wqrst$ is such that $\rho(\mu)=\alpha$ and let
\begin{align*}
\rho(x)&=\chi, & \rho(y)&=\psi, & \rho(x^a)&=\rho(y^b)=\eps. 
\end{align*}
If $\langle\hat\chi,\hat\psi\rangle$ is a cyclic subgroup of \dihedral{s}t\ then  $\langle\chi,\psi\rangle$ is itself a cyclic subgroup of \Wqrst, and $\eps = \alpha^{ab}$. Otherwise, if $\langle\hat\chi,\hat\psi\rangle$ 
is a non-cyclic subgroup of \dihedral{s}t\ then
\begin{enumerate} 
\item\label{item:eps.shape}
$\hat\eps = 1$;
\item\label{item:eps.constant.on.orbits}
$\eps_g$ is constant on orbits of $\hat\alpha$;
\item\label{item:eps.conjugacy}
the conjugacy class of $\eps_g$ is constant on $\Cst$;
\item\label{item:eps.modr}
$[\eps_g]= \frac {ab}{s^{\ord_t(s)}} [[\alpha]]\in\integer_r$ for all $g\in\Cst$;
\item\label{item:eps.value}
$\eps_g = \xi^{ab[[\alpha]]/s^{\ord_t(s)}}$ if $\alpha_g \neq 1$.
\end{enumerate}
\end{lemma}

\begin{proof}
We consider the cases where $\langle\hat\chi,\hat\psi\rangle$ is a cyclic and non-cyclic subgroup of \dihedral{s}t\ in turn. 

\subsubsection{The cyclic case} 

If $\langle\hat\chi,\hat\psi\rangle$ is a cyclic subgroup of \Dst, then the homomorphism $\hat\rho:G_{a,b}\to\Dst$ given by $\hat\rho(\omega)=\widehat{\rho(\omega)}$ factors through the abelianisation $G_{a,b}\to\integer$. The abelianisation is generated by the image of the meridian, so
\begin{align*}
\hat\chi&=\hat\alpha^b, & \hat\psi&=\hat\alpha^a, & \hat\eps&=\hat\alpha^{ab}. 
\end{align*}
Now every nontrivial element of \Dst\ has order $s$ or $t$, and $s|a$ and $t|b$, so $\hat\eps=1$ and either $\hat\chi=1$ or $\hat\psi=1$. 

The construction of $\Dst$ is not symmetric in $s$ and $t$, so we cannot simply assume without loss of generality that $\hat\chi=1\in\Dst$. Nevertheless, the arguments in both cases are entirely analogous so we just present one of them. 
Suppose then that $\hat\chi=1$ (that is, suppose that $\hat\alpha$ has order $t$). Then the entries of $\eps=\chi^a$ are simply given by
\[
\eps_g = \chi_g^a
\]
for each $g\in\Cst$. The primes $q$ and $r$ are coprime to $a$, so there are integers $c$ and $d$ such that $ac\equiv1\bmod q$, and $ad\equiv 1\bmod r$, and then by the Chinese Remainder Theorem there is an integer $k$ such that $k\equiv c\bmod q$ and $k\equiv d\bmod r$. Then $ak\equiv1\bmod qr$.  Noting that $\chi_g$ has order 1, $q$ or $r$ we then get
\[
\eps_g^k = (\chi_g^a)^k = \chi_g^{ak} = \chi_g,
\]
and since $\hat\eps=\hat\chi=1$ we get $\chi=\eps^k$. 

Recall that $x^a=y^b$ generates the centre of $G_{a,b}$. Thus, $\eps$ commutes with $\psi$, and so $\chi=\eps^k$ commutes with $\psi$ also. Therefore $\rho(G_{a,b})=\langle\chi,\psi\rangle$ is abelian, so $\rho$ itself factors through the abelianisation $G_{a,b}\to\integer$. Since the abelianisation is generated by the meridian we have
\begin{align*}
\chi&=\alpha^b, & \psi&=\alpha^a, & \eps&=\alpha^{ab}. 
\end{align*}

\subsubsection{The non-cyclic case} 

Suppose that $\langle\hat\chi,\hat\psi\rangle$ is a non-cyclic subgroup of \Dst. Then since $\hat\chi^a=\hat\psi^b=\hat\eps$ and $\langle\hat\chi,\hat\psi\rangle$ is a non-cyclic subgroup of \Dst, 
by Lemma~\ref{lem:non-cyclicsolutions} it must be the case that $\hat\chi$ has order $s$, $\hat\psi$ has order $t$, and $\hat\eps=1\in\Dst$.
Then $\alpha_g$ is constant on the orbits of $\hat\eps$ (since each orbit is a singleton), and moreover $\alpha$ commutes with $\eps$ (since $x^a=y^b$ generates the centre of $G_{a,b})$. We may therefore apply Lemma~\ref{lem:rsf-centraliser} to conclude that $\eps_g$ commutes with $\alpha_g$ for all $g\in\Cst$, and $\eps_g$ is constant on orbits of $\hat\alpha$.
This proves parts~\eqref{item:eps.shape} and~\eqref{item:eps.constant.on.orbits}. 

To prove part~\eqref{item:eps.conjugacy} we apply Lemma~\ref{lem:mthpower} to $\eps=\chi^a=\psi^b$. Each orbit of $\hat\eps$ acting on $\Cst$ is a singleton, so $\eps_g=\pi_g(\eps)$ for all $g$ and it follows that the conjugacy class of $\eps_g$ is constant on the orbits of both $\hat\chi$ and $\hat\psi$. But $\langle\hat\chi,\hat\psi\rangle=\Dst$ by Lemma~\ref{lem:dihedral.generators} and \Dst\ acts transitively on \Cst, so the conjugacy class of $\eps_g$ must be constant on \Cst. 

Since the conjugacy class of $\eps_g$ is constant on \Cst, so is the value of $[\eps_g]\in\integer_r$. To prove part~\eqref{item:eps.modr} we evaluate this common value using the abelianisation $[[\cdot]]:\Wqrst\to\integer_r$. On one hand we have
\[
[[\eps]]= \sum_{h\in\Cst} [\eps_h]= |\Cst|\cdot[\eps_g]=s^{\ord_t(s)}[\eps_g], 
\]
and on the other we have
\[
[[\eps]]=[[\chi^a]]=a[[\chi]]=ab[[\alpha]],
\]
because $\alpha$ is the image of the meridian, which generates the abelianisation of $G_{a,b}$. Therefore 
$s^{\ord_t(s)}[\eps_g]=ab[[\alpha]]$ in $\integer_r$, and since $s$ is coprime to $r$ we may divide by $s^{\ord_t(s)}$ to get $[\eps_g]=ab[[\alpha]]/s^{\ord_t(s)}$ in $\integer_r$. 

Finally, to prove part~\eqref{item:eps.value}, recall that $\alpha_g\in\langle\xi\rangle$ for all $g$, and we have shown above that $\eps_g$ commutes with $\alpha_g$ for all $g$. By Lemma~\ref{lem:dihedralprops} nontrivial elements of $\langle\xi\rangle$ commute only with other elements of $\langle\xi\rangle$, so if $\alpha_g\neq1$ then $\eps_g\in\langle\xi\rangle$ also. Since $[\xi]=1$ we must have 
$\eps_g = \xi^{ab[[\alpha]]/s^{\ord_t(s)}}$ if $\alpha_g \neq 1$. 
\end{proof}

\section{Proof of the Main Theorem}
\label{sec:proof}

We now have all the ingredients we require to prove  our main result by proving Theorem~\ref{thm:main-refined}. Recall that we will do this by proving the following two statements:
\begin{enumerate}
\renewcommand{\theenumi}{\Roman{enumi}}
\item
Every map-root pair for $H_{a,b}$ in \Wqrst\ that is compatible for $\GK{a}{b}$ is also compatible for $\SK{a}{b}$.
\item
There exist map-root pairs for $H_{a,b}$ in \Wqrst\ that are compatible for $\SK{a}{b}$ but not $\GK{a}{b}$.
\end{enumerate}
The compatibility conditions are
\begin{align*}
\GK{a}b:& &
\rho(x^a)\rho(w^a)\eta &= \eta\rho(x^{a})\rho(w^a) & \text{or}&&
\rho(x^{-a})\eta\rho(x^{a}) &= \rho(w^a)\eta\rho(w^{-a}), \\
\SK{a}b:&&
\rho(x^a)\rho(w^{-a})\eta &= \eta\rho(x^{a})\rho(w^{-a}) & \text{or}&&
\rho(x^{-a})\eta\rho(x^{a}) &= \rho(w^{-a})\eta\rho(w^{a}).
\end{align*}
To begin, let $(\rho,\eta)$ be a map-root pair for $H_{a,b}$ in \Wqrst, and let $\hat\rho:H_{a,b}\to\Dst$ be the induced map defined by 
$\hat\rho(\omega)=\widehat{\rho(\omega)}$ for all $\omega\in H_{a,b}$. Our initial case division is according to whether or not $\hat\rho$ is trivial. Since the conjugacy class of $\mu$ generates $H_{a,b}$, the map $\hat\rho$ is trivial if and only if $\hat\rho(\mu)=1\in\Dst$. 

\subsection{Trivial induced maps to \Dst}

Suppose that $\hat\rho$ is trivial. Then we may regard the homomorphism $\rho$ as a homomorphism $H_{a,b}\to(\Dqr)^\Cst$, and as such it is a product of homomorphisms $\rho_g:H_{a,b}\to\Dqr$. Since $\hat\eta^n=\hat\rho(\mu)=1\in\Dst$, and $\gcd(st,n)=1$, we must have $\hat\eta=1\in\Dst$ also. It follows that $(\rho,\eta)$ decomposes as a collection of map-root pairs $\{(\rho_g,\eta_g):g\in\Cst\}$ for $H_{a,b}$ in \Dqr.

By hypothesis $ab$ is coprime to $qr$, so by Lemma~\ref{lem:non-cyclicsolutions} any homomorphism $G_{a,b}\to\Dqr$ has cyclic image. Consequently $\rho_g(H_{a,b})$ is cyclic too, generated by $\alpha_g=\rho_g(\mu)$. Then
\[
\rho_g(x^a)=\rho_g(w^a)=\alpha_g^{ab}=\eta_g^{abn},
\]
so both $\rho_g(x^a)$ and $\rho_g(w^a)$ commute with $\eta_g$. It follows that $(\rho_g,\eta_g)$ is compatible for both $\GK{a}b$ and $\SK{a}b$ for all $g\in\Cst$, and hence that $(\rho,\eta)$ is compatible for both $\GK{a}b$ and $\SK{a}b$. 

\subsection{Nontrivial induced maps to \Dst}

Now suppose that $\hat\rho$ is nontrivial. Then $\rho(\mu)=\eta^n$ is an $n$th power in \Wqrst\ such that $\widehat{\rho(\mu)}\neq1$, so by Lemma~\ref{lem:meridian} it is conjugate to an element $\alpha$ of $\Arst$ in reduced standard form. Let $\beta\in\Wqrst$ be such that $\beta\rho(\mu)\beta^{-1}=\alpha$, set $\eta'=\beta\eta\beta^{-1}$, and define $\rho':H_{a,b}\to\Wqrst$ by $\rho'(g)=\beta\rho(g)\beta^{-1}$.  Then 
$(\rho',\eta')$ is a map-root pair for $H_{a,b}$ in \Wqrst\ such that $\rho'(\mu)=\alpha$, and $(\rho,\eta)$ is compatible for $\GK{a}b$ or $\SK{a}b$ if and only if 
$(\rho',\eta')$ is. So it suffices to prove statement~\eqref{item:inequality} under the assumption that $\rho(\mu)=\alpha$ is an element of \Arst\ in reduced standard form. 

Let $\eps = \rho(x^a)$, $\delta = \rho(w^a)$. Then $\eps$ and $\delta$ are each described by Lemma~\ref{lem:longitude}, and the compatibility conditions may be written as
\begin{align*}
\GK{a}b:&&\eta\eps\delta &=\eps\delta\eta &\text{or}&&
\eps^{-1}\eta\eps &= \delta\eta\delta^{-1}, \\
\SK{a}b:&&  \eta\eps\delta^{-1} &=\eps\delta^{-1}\eta &\text{or}&&
\eps^{-1}\eta\eps &= \delta^{-1}\eta\delta.
\end{align*}
We consider two cases:

\subsubsection{At least one of $\eps$ or $\delta$ is equal to $\alpha^{ab}$}

Without loss of generality suppose that $\delta=\alpha^{ab}=\eta^{abn}$. Then $\delta$ commutes with $\eta$, so for both knots the compatibility condition is that $\eps$ commutes with $\eta$ also. 
Thus $(\rho,\eta)$ is compatible for $\GK{a}b$ if and only it is compatible for $\SK{a}b$.

\subsubsection{Neither $\eps$ nor $\delta$ is equal to $\alpha^{ab}$}

If $\eps\neq\alpha^{ab}\neq\delta$ then $\eps$ and $\delta$ are both described by statements~\eqref{item:eps.shape}--\eqref{item:eps.value} of Lemma~\ref{lem:longitude}, and $\eta$ is described by Lemma~\ref{lem:meridian}. We will make use of the following lemma, with $\beta$ equal to $\eta$, and $\gamma$ equal to $\eps\delta$ and $\eps\delta^{-1}$ in turn:
\begin{lemma}
\label{lem:centralising}
Suppose that $\beta,\gamma\in\Wqrst$ are such that $\hat\beta,\hat\gamma\in\Dst$ commute, and $\beta_{g\cdot\hat\gamma}=\beta_g$, $\gamma_{g\cdot\hat\beta}=\gamma_g$ for all $g\in\Cst$. Then $\beta\gamma=\gamma\beta$ if and only if $\beta_g\gamma_g=\gamma_g\beta_g$ for all $g\in\Cst$. 
\end{lemma}

In particular, the hypotheses of Lemma~\ref{lem:centralising} are satisfied if $\hat\gamma=1$ and $\gamma_{g\cdot\hat\beta}=\gamma_g$ for all $g\in\Cst$. We will use the lemma in this special case. 

\begin{proof}
Since $\hat\beta$ commutes with $\hat\gamma$ we have $\widehat{\beta\gamma}=\hat\beta\hat\gamma=\hat\gamma\hat\beta=\widehat{\gamma\beta}$. Therefore $\beta\gamma=\gamma\beta$ if and only if $(\beta\gamma)_g=(\gamma\beta)_g$ for all $g\in\Cst$. On the one hand 
since $\gamma_{g\cdot\hat\beta}=\gamma_g$ for all $g$ we have
\[
(\beta\gamma)_g = \beta_g\gamma_{g\cdot\hat\beta} = \beta_g\gamma_g.
\]
On the other, since $\beta_{g\cdot\hat\gamma}=\beta_g$ for all $g$ we have
\[
(\gamma\beta)_g = \gamma_g\beta_{g\cdot\hat\gamma}=\gamma_g\beta_g.
\]
The result follows. 
\end{proof}

Since $\hat\eps=\hat\delta=1\in\Dst$ we have $\widehat{\eps\delta}=\widehat{\eps\delta^{-1}}=1\in\Dst$ also, and
\begin{align*}
(\eps\delta)_g&= \eps_g\delta_{g\cdot\hat\eps} = \eps_g\delta_g, \\
(\eps\delta^{-1})_g &= \eps_g(\delta^{-1})_{g\cdot\hat\eps}
= \eps_g(\delta^{-1})_g= \eps_g(\delta_{g\cdot\hat\delta^{-1}})^{-1}
= \eps_g(\delta_g)^{-1}.
\end{align*}
Then by Lemma~\ref{lem:longitude} part~\eqref{item:eps.constant.on.orbits} we have
\begin{align*}
(\eps\delta)_{g\cdot\hat\eta} 
    &= \eps_{g\cdot\hat\eta}\delta_{g\cdot\hat\eta}
     = \eps_{g\cdot\hat\alpha^k}\delta_{g\cdot\hat\alpha^k}
     = \eps_g\delta_g = (\eps\delta)_g,\\
(\eps\delta^{-1})_{g\cdot\hat\eta} 
    &= \eps_{g\cdot\hat\eta}(\delta_{g\cdot\hat\eta})^{-1}
     = \eps_{g\cdot\hat\alpha^k}(\delta_{g\cdot\hat\alpha^k})^{-1}
     = \eps_g(\delta_g)^{-1} = (\eps\delta^{-1})_g,
\end{align*}
so Lemma~\ref{lem:centralising} applies with $\beta=\eta$ and $\gamma$ equal to either $\eps\delta$ or $\eps\delta^{-1}$.

We now check when $\eta_g$ commutes with each of $(\eps\delta)_g$ and $(\eps\delta^{-1})_g$. We consider two cases, according to whether or not $\alpha_g=1$:
\begin{enumerate}
\item
If $\alpha_g\neq 1$ then $\eta_g$, $\eps_g$ and $\delta_g$ all belong to $\langle\xi\rangle$, and so $\eta_g$ commutes with both $(\eps\delta)_g$ and $(\eps\delta^{-1})_g$. 
(In fact we have $\eps_g=\delta_g=\xi^{ab[[\alpha]]/s^{\ord_t(s)}}$, so
\begin{align*}
(\eps\delta)_g &=\xi^{2ab[[\alpha]]/s^{\ord_t(s)}}, &
(\eps\delta^{-1})_g &= 1.)
\end{align*}

\item
If $\alpha_g=1$ then $\eta_g\in\Vee{q}r$. 
Applying $[\cdot]:\Dqr\to\integer_r$ to each of $(\eta\delta)_g$ and $(\eps\delta^{-1})$ we get
\begin{align*}
[(\eps\delta)_g] &= [\eps_g]+[\delta_g]=\frac{2ab[[\alpha]]}{s^{\ord_t(s)}}, \\
[(\eps\delta^{-1})_g] &=[\eps_g]-[\delta_g] = 0,
\end{align*}
and since $r$ is chosen coprime to $2ab$ it follows that $(\eps\delta^{-1})_g\in\Vee{q}r$ for all $g$, but 
$(\eps\delta)_g\in\Vee{q}r$ if and only if $[[\alpha]]=0\in\integer_r$. It follows that $(\eps\delta^{-1})_g$ always commutes with $\eta_g$, but $(\eps\delta)_g$ only commutes with $\eta_g$ if $\eta_g=1$, or $[[\alpha]]=0$. 
\end{enumerate}

We conclude that when $\eps\neq\alpha^{ab}\neq\delta$, every map-root pair is compatible for $\SK{a}b$, but $(\rho,\eta)$ is compatible for $\GK{a}b$ only if $[[\alpha]]=0$, or if $\eta_g=1$ whenever $\alpha_g=1$. 
In the next section, we will complete the proof of the theorem by showing that there exists $\rho$ realising $[[\alpha]]\neq0$, together with $\alpha_g=1$ for some $g\in\Cst$.

\subsection{Realisation}

Choose $\hat\chi,\hat\psi\in\Dst$ arbitrarily such that $\hat\chi$ has order $s$ and $\hat\psi$ has order $t$. Then since $s|a$ and $t|b$ there is a homomorphism $\hat\rho:G_{a,b}\to\Dst$ such that $\hat\rho(x)=\hat\chi$ and $\hat\rho(y)=\hat\psi$.

By Theorem~\ref{thm:Gab} the meridian of $T_{a,b}$ in $G_{a,b}$ is given by $\mu=y^dx^{-c}$, where $ad-bc=1$. Both $c$ and $d$ may be chosen to be positive, and we observe that $\gcd(c,s)=\gcd(d,t)=1$, which implies that $\hat\chi^c\neq1\neq\hat\psi^d$. Therefore $\hat\chi^c$ has order $s$, and $\hat\psi^d$ has order $t$. Let $\hat\beta=\hat\rho(\mu)=\hat\psi^d\hat\chi^{-c}$, and note that $\hat\beta$ must have order $t$, because $\hat\chi^c$ belongs to $\Vee{s}t$ but $\hat\psi^d$ does not. 

Since $\hat\beta$ has order $t$, by Remark~\ref{rem:cyclestructure} its action on $\Cst$ has a unique fixed point $f$ (namely, whichever power of $\hat\beta$ lies in $\Cst$). The fixed point $f$ cannot be the fixed point of $\hat{\psi}$, because then $f$ would be fixed by $\hat{\chi}^c=\hat\beta^{-1}\hat{\psi}^d$, which acts freely.  
Extend $\hat\chi,\hat\psi$ to elements of $\Wqrst$ by defining $\chi_g = \xi^b$ for all $g\in\Cst$, and 
\[
\psi_g = 
\begin{cases}
\xi^{a+1},      &\text{if $g = f\cdot\hat\psi^{-1}$;}\\
\xi^{a-1},       &\text{if $g = f$;}\\
\xi^a,             &\text{otherwise}.
\end{cases}
\]
The cycle of $\hat\psi$ acting on $f$ is illustrated in Figure~\ref{fig:cycle} in the case $t=7$. 
Then it is seen (using Figure~\ref{fig:cycle} for the cycle of $\hat{\psi}$ containing $f$) that $\pi_g(\psi)=\xi^{at}$ for all $g$ 
not the fixed point of $\hat{\psi}$. It then follows that
\begin{align*}
(\chi^a)_g=(\psi^ b)_g= \xi^{ab}
\end{align*}
for all $g\in\Cst$. Since also $\hat{\chi}^a=\hat{\psi}^b=1$, we have $\chi^a=\psi^b$, and therefore $\rho(x)=\chi, \rho(y)=\psi$ defines a homomorphism $\rho:G_{a,b} \to \Wqrst$. We may then obtain a homomorphism $\rho:H_{a,b} \to \Wqrst$ by setting $\rho(w)=\chi, \rho(z)=\psi$ also.

We now calculate $\chi^{-c}$ and $\psi^d$ in preparation for computing $\beta=\rho(\mu)=\psi^{d}\chi^{-c}$. Since $\chi_g=\xi^b$ for all $g$ we have simply $(\chi^{-c})_g = \xi^{-bc}$ for all $g$; and referring again to Figure~\ref{fig:cycle}, we have $(\psi^d)_g=\xi^{ad}$ except at those points where the product
\[
(\psi^d)_g = \prod_{k=0}^{d-1}\psi_{g\cdot\hat\psi^k}
\]
begins at $f$ or ends at $f\cdot\hat\psi^{-1}$. The latter point is where $g\cdot\hat\psi^{d-1}=f\cdot\hat\psi^{-1}$, so $g=f\cdot\hat\psi^{-d}$ and we have
\[
(\psi^{d})_g= 
\begin{cases}
\xi^{ad-1},      &\text{if $g = f$;}\\
\xi^{ad+1},       &\text{if $g = f\cdot\hat{\psi}^{-d}$;}\\
\xi^{ad},             &\text{otherwise}.
\end{cases}
\]
Then
\[
\beta_g = (\psi^{d}\chi^{-c})_g 
= (\psi^d)_{g}(\chi^{-c})_{g\cdot\hat\psi^d}
=  (\psi^{d})_{g}\xi^{-bc}, 
\]
so
\[
\beta_g= 
\begin{cases}
\xi^{ad-bc-1}=\xi^0=1,      &\text{if $g=f$;}\\
\xi^{ad-bc+1}=\xi^2,        &\text{if $g=f\cdot\hat\psi^{-d}$;}\\
\xi^{ad-bc}=\xi,                        &\text{otherwise}.
\end{cases}
\]
Note that $[[\beta]]=s^{\ord_t(s)}\not\equiv0\bmod r$, because $\gcd(r,s)=1$. By Lemmas~\ref{lem:rsf} and~\ref{lem:inArst} there exists $\omega\in\Wqrst$ such that $\alpha=\omega\beta\omega^{-1}$ is an element of $\Arst$ in reduced standard form, with $\hat\alpha=\hat\beta$ and
\[
\alpha_g= 
\begin{cases}
1,                          &\text{if $g=f$;}\\
\xi^{(t+1)/t},        &\text{if $g\in(f\cdot\hat\psi^{-d})\cdot\langle\hat\alpha\rangle$;}\\
\xi,            &\text{otherwise}.
\end{cases}
\]
We still have $[[\alpha]]\neq0$, so the homomorphism $\rho':H_{a,b}\to\Wqrst$ defined by $\rho'(g)=\omega\rho(g)\omega^{-1}$ realises the case where there are $n$th roots of $\rho'(\mu)=\alpha$ that are compatible for $\SK{a}b$ but not $\GK{a}b$, namely, pairs $(\rho',\eta)$ where $\eta_f\neq1$. This establishes statement~\eqref{item:strict}, and completes the proof of Theorem~\ref{thm:main-refined}.

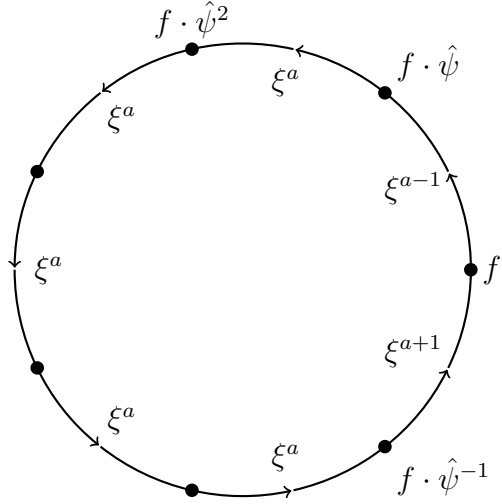
\begin{figure}
\begin{center}
\begin{tikzpicture}[vertex/.style={circle,draw,fill=black,minimum size = 1.5mm,inner sep=0pt},thick]
\setlength{\radius}{3cm}
\foreach \x in {0,1,...,6}
\node (\x) at (\x*360/7:\radius) [vertex] {};
\foreach \x in {0,1,...,6}
\node (mid\x) at (\x*360/7-180/7:\radius) {};
\foreach \x in {0,1,...,6} 
\draw[->] (mid\x) arc ((\x*360/7-180/7:(\x+0.49)*360/7:\radius);
\foreach \x/\y/\p in {6/{$f\cdot\hat\psi^{-1}$}/{below right},0/{$f$}/right,1/{$f\cdot\hat\psi$}/{above right},2/{$f\cdot\hat\psi^2$}/{above}}
   \node [\p] at (\x) {\y};
\foreach \x in {1,2,...,5}
   \node at (\x*360/7+180/7:0.85*\radius) {$\xi^a$};
\node at (180/7:0.83*\radius) {$\xi^{a-1}$};
\node at (-180/7:0.83*\radius) {$\xi^{a+1}$};
\end{tikzpicture}
\caption{Diagram of the cycle of $\hat\psi$ containing $f$ in the case where $t=7$.} 
\label{fig:cycle}
\end{center}
\end{figure}

\end{document}